\documentclass{amsart}
\usepackage{csvsimple}

\usepackage{amssymb}
\usepackage{amsthm}
\usepackage{graphicx}
\usepackage{hyperref}
\usepackage{multirow}
\usepackage{array}
\usepackage[colorinlistoftodos]{todonotes}
\usepackage[english]{babel}
\usepackage{lmodern}

\usepackage[utf8]{inputenc}
\usepackage[T1]{fontenc}
\usepackage{comment}

\newcommand{\modulo}[3]{#1\equiv#2\textrm{ }(\textrm{mod }#3)}
\newtheorem{theorem}{Theorem}[section]
\newtheorem{lemma}[theorem]{Lemma}
\newtheorem{corollary}[theorem]{Corollary}

\theoremstyle{definition}

\newtheorem{proposition}[theorem]{Proposition}

\newtheorem{example}[theorem]{Example}

\newtheorem{question}[theorem]{Question}

\theoremstyle{remark}
\newtheorem{remark}[theorem]{Remark}
\numberwithin{equation}{section}

\title[On higher congruences.II.]{On higher congruences between cusp forms and Eisenstein series. II.}
\author{Bartosz Naskręcki}

\address{Faculty of Mathematics and Computer Science, Adam Mickiewicz University\\
	Umultowska 87, 61-614 Poznań, Poland\\
	E-mail: nasqret@gmail.com}

\begin{document}
	\begin{abstract}
	We study congruences between cuspidal modular forms and Eisenstein series at levels which are square-free integers and for equal even weights. This generalizes our previous results from \cite{Naskrecki_Springer} for prime levels and provides further evidence for the sharp bounds obtained under restrictive ramification conditions. We prove an upper bound on the exponent in the general square-free situation and also discuss the existence of the congruences when the coefficients belong to the rational numbers and weight equals $2$.
\end{abstract}
	\maketitle
	\tableofcontents

\section{Introduction}

Let $N$ be a square-free positive integer and $f$ be a cuspidal newform of level $\Gamma_{0}(N)$ and even weight $k$. We consider in this paper congruences between Fourier coefficients $a_{n}(f)$ of the newform $f$ and coefficients $a_{n}(E)$ of a suitably normalized Eisenstein series $E$ of weight $k$ modulo prime powers of an ideal $\lambda$ lying in the coefficient field $K_{f}$ of $f$ and a positive integer $r$
\begin{equation}\label{eq:main_congruence}
\modulo{a_{n}(f)}{a_{n}(E)}{\lambda^{r}}
\end{equation}
for all $n\geq 0$. In the classical setting of $N$ prime and weight $k=2$ the existence of such congruences for $E=E_{2}$ was established by Mazur in \cite{Mazur_modular}. 

In this paper we focus on the extension of the results and computations of \cite{Naskrecki_Springer}. We present a summary of a large volume of computations performed with MAGMA and a construction of the algorithm that is used to find congruences of type \eqref{eq:main_congruence}. We also present a partial classification of congruences when  $\mathcal{O}_{f}=\mathbb{Q}$. The existence of the higher weight congruences for prime levels was discussed in \cite{Dummigan_Fretwell}. Sufficient conditions for existence of congruences for composite levels in weight $2$ were obtained by Yoo in \cite{Yoo_thesis}, \cite{Yoo_multiplicity}. The lower bounds on the congruence exponent in weight $2$ was discussed by Hsu \cite{Hsu} and Berger-Klosin-Kramer \cite{BKK}. Congruences on the level of Galois representations were studied recently by Billerey and Menares in \cite{Billerey_Menares_2} and \cite{Billerey_Menares_1}. Lecouturier \cite{Lecouturier} computed the rank $g_p$ of the completion of the Hecke algebra acting on cuspidal modular forms of weight $2$ and level $\Gamma_{0}(N)$ at the $p$-maximal Eisenstein ideal using the knowledge of the exponent of the Eisenstein congruence.

In \cite{Martin} Martin proved the existence of congruences between cuspidal and Eisenstein modular forms for levels which are not square-free. It would be interesting to find more computational examples of such congruences, which seem to be very sparsely distributed.

\medskip                               
With the notation of Section \ref{sec:Standard_basis} we present the main theorems of this paper.
{
	\renewcommand{\thetheorem}{\ref{corollary:bound_for_a_0_nonzero}}
\begin{theorem}
Let $p_{1},\ldots,p_{t}$ be different prime factors of $N$ a square-free integer and let $k>2$. Suppose that $f\in\mathcal{S}_{k}(N)^{\mathrm{new}}$ is a newform which is congruent to the Eisenstein eigenform $E=[p_{1}]^{+}\circ\ldots\circ[p_{t}]^{+}E_{k}\in\mathcal{E}_{k}(N)$ modulo a power $r>0$ of a maximal ideal $\lambda\subset\mathcal{O}_{f}$. If $\ell$ is the residual characteristic of $\lambda$, we obtain the bound
\[r\leq ord_{\lambda}(\ell)\cdot v_{\ell}\left(-\frac{B_{k}}{2k} \prod_{i=1}^{t}(1-p_{i})\right). \]
\end{theorem}
\addtocounter{theorem}{-1}
}
The upper bound predicted by the theorem above is discussed numerically in Section \ref{sec:Numerical_data}. It turns out that in general it is optimal but in the case when the ideal $\lambda$ is ramified above $\ell$ the ramification degree $ord_{\lambda}(\ell)$ seems be the right upper bound in most cases.

When we turn to the eigenforms with weight $2$ and rational coefficients the existence of the congruences between cuspidal newforms and Eisenstein series is limited to only finitely many prime powers. In this case we investigated modular forms of levels with two prime factors.
{
	\renewcommand{\thetheorem}{\ref{thm:congruences_over_Q}}
\begin{theorem}
Let $p$,$q$ be two different primes. Suppose that $f\in\mathcal{S}_{2}(pq)^{\mathrm{new}}$ is a newform with rational coefficients and let $E$ be an eigenform in $\mathcal{E}_{2}(pq)$. Let $\ell$ be a prime number and $r>0$ an integer such that the congruence (\ref{equation:rational_congruences}) holds for all $n\geq 0$. Then one of two conditions holds:
\begin{itemize}
\item[(1)] $\ell^{r}\in\{2,3,4,5\}$ or
\item[(2)] $\ell^{r}=7$ and $E=[13]^{-}[2]^{+}E_{2}$.
\end{itemize}
\end{theorem}
\addtocounter{theorem}{-1}
}
Our numerical results discussed in Section \ref{sec:Computational_results} lead to certain further speculations about the upper bound for the congruences of type \eqref{eq:main_congruence}.

\begin{question}
	Suppose that the congruence \eqref{eq:main_congruence} holds for a prime ideal $\lambda$ above a rational prime $\ell$ and an exponent $r>0$. Let $e$ denote the order $ord_{\lambda}(\ell)$ and suppose that $e>1$. For $k=2$ and level $N$ prime we checked that $r\leq e$ for every prime $\ell>3$ for $N\leq 13009$. In a similar fashion, for any square-free integer $N$ and a weight $k$ described in Table~\ref{table:range} the same conclusion holds except for two counterexamples found only in weight $2$, described in Table~\ref{table:exceptional_ramification}. Based on those observations we ask the following two questions:
	\begin{itemize}
		\item Assume that $e>1$, the level $N$ is prime and weight $k=2$. Is it true that for every prime number $\ell>3$  a congruence of type \eqref{eq:main_congruence} satisfies the condition $r\leq e$?
		\item Assume that $e>1$, the level $N$ is square-free and weight $k\geq 2$ is even. Is it true that for \textit{every fixed} prime number $\ell$ there are only finitely many congruences of type \eqref{eq:main_congruence} which satisfy the condition $r>e$?
	\end{itemize}
The questions above are discussed in detail in Section \ref{sec:Computational_results}.
\end{question}

\subsection*{Summary of the paper} In Section \ref{sec:Standard_basis} we describe a standard basis for the Eisenstein subspace of modular forms of square-free level that consists of the eigenforms. The material in this section is rather classical but we did not find a convenient reference which contained all the necessary results. In Section \ref{sec:Upper_bound_of_congruences} we prove Theorem \ref{corollary:bound_for_a_0_nonzero} using the results of Atkin-Lehner from \cite{Atkin_Lehner} as our main tool. In Section \ref{sec:Rational_cong} we study the congruence between cuspidal eigenforms of weight $2$ and with rational coefficients and Eisenstein eigenforms. A result of Katz \cite{Katz_Torsion} and the theorem of Mazur \cite{Mazur_modular} allow us to conclude that there are only finitely many prime powers for which the congruences exist. A refined statement is proved in Theorem \ref{thm:congruences_over_Q}. In Sections \ref{sec:Algorithms}, \ref{sec:Numerical_data}, \ref{sec:Computational_results} we described an improved version of the algorithm that finds congruences of desired shape for fixed levels and weights (cf. \cite{Naskrecki_Springer}). We then discuss the numerical results contained in the attached tables and formulate some of them as corollaries from the computations. A complete database of congruences is available on request.

\subsection*{Notation}
Let $B_{k}$ denote the $k$-th Bernoulli with $B_{1}=-\frac{1}{2}$ and let $\sigma_{k-1}(n)=\sum_{m\mid n} m^{k-1}$ denote the divisor function for any integer $k\geq 2$. Let $E_{k}=-B_{k}/(2k)+\sum_{n=1}^{\infty}\sigma_{k-1}(n)q^n$ denote the Eisenstein series of weight $k$ and level $1$, where $q=\exp(2\pi i \tau)$ for $\tau$ in the upper half-plane $\mathcal{H}$. Let $$\Gamma_{0}(N)=\left\{ \left(\begin{array}{cc}
a & b\\ c & d
\end{array}\right)\in \textrm{SL}_{2}(\mathbb{Z}): N\mid c\right\}$$ denote the Hecke congruence subgroup of level $N$ and $\mathcal{M}_{k}(N)$ the space of modular forms of weight $k$ and level $N$ with respect to the group $\Gamma_{0}(N)$. Let $\mathcal{S}_{k}(N)$, $\mathcal{E}_{k}(N)$ and $\mathcal{S}_{k}(N)^{\mathrm{new}}$ denote respectively the subspace of cusp forms, the subspace of Eisenstein series and the subspace of newforms in $\mathcal{M}_{k}(N)$. On $\mathcal{M}_{k}(N)$ we have the action of the Hecke algebra $\mathbb{T}_{N}$ where $T_{p}$ is the Hecke operator with index $p$, $p\nmid N$ and $U_{p}$ if $p\mid N$. Let $a_{n}(f)$ denote the $n$-th Fourier coefficient of $f$ expanded at infinity. 
\section{Standard basis of Eisenstein eigenforms}\label{sec:Standard_basis}
In this section we present a convenient basis of Eisenstein eigenforms in $\mathcal{E}_{k}(N)$ for all $k\geq 2$ with respect to the Hecke algebra $\mathbb{T}_{N}$. We believe that the presented material is not new, however due to a lack of complete reference we present full proofs here. Let us denote by $A_{d}$ a linear map from $\mathcal{M}_{k}(N)$ to $\mathcal{M}_{k}(Nd)$  such that $A_{d}:f(\tau)\mapsto f(d\tau)$. The operator $A_{d}$ is just a normalized slash operator $A_{d}(f)=d^{1-k} f\mid_{k}\gamma$ where
\[\gamma=\left(\begin{array}{cc}
d & 0 \\
0 & 1
\end{array}\right).\]
We quote now a theorem of Atkin-Lehner which will be used at several places.
\begin{lemma}[{\cite[Lemma 15]{Atkin_Lehner}}]
Let $f$ be a modular form in $\mathcal{M}_{k}(N)$. We have the following relation between different Hecke operators acting on $f$
\begin{align}
(T_{q}\circ U_{p})(f)=(U_{p}\circ T_{q})(f)& \quad \textrm{for }p\neq q,\label{align:atkin_prop1}\\
(T_{q}\circ A_{d})(f)=(A_{d}\circ T_{q})(f)& \quad \textrm{for }(q,d)=1,\label{align:atkin_prop2}\\
(U_{q}\circ A_{d})(f)=(A_{d}\circ U_{q})(f)& \quad \textrm{for }(q,d)=1.\label{align:atkin_prop3}
\end{align}
\end{lemma}

For any $k>2$ the series $E_{k}$ is an eigenform in $\mathcal{M}_{k}(1)$ with respect to the Hecke algebra $\mathbb{T}_{1}$. In particular for any $T_{n}$ acting on $\mathcal{M}_{k}(1)$ for $k>2$ we have
\begin{equation}\label{equation:T_n_on_E_k}
T_{n}(E_{k})=a_{n}(E_{k})E_{k}=\sigma_{k-1}(n)E_{k}. 
\end{equation}
We also record three simple identities related to $\sigma_{k}$ functions. Let $n$ be a positive integer and $p$ a prime number such that $p\mid n$. For any $k\geq 2$ we have
\begin{align}
\sigma_{k-1}(np)+p^{k-1}\sigma_{k-1}(n/p)&=\sigma_{k-1}(p)\sigma_{k-1}(n),\label{align:sigma1}\\
\sigma_{k-1}(n)-p^{k-1}\sigma_{k-1}(n/p)&=\sigma_{k-1}(np)-p^{k-1}\sigma_{k-1}(n),\label{align:sigma2}\\
\sigma_{k-1}(np)-\sigma_{k-1}(n)&=p^{k-1}(\sigma_{k-1}(n)-\sigma_{k-1}(n/p)).\label{align:sigma3}
\end{align}
For a fixed positive integer $d$ we define two additional linear operators
\begin{align*}
[d]^{+}:=T_{1}-d^{k-1}A_{d} &: \mathcal{M}_{k}(N)\rightarrow\mathcal{M}_{k}(Nd), \\
[d]^{-}:=T_{1}-A_{d} &: \mathcal{M}_{k}(N)\rightarrow\mathcal{M}_{k}(Nd).
\end{align*}
where $T_{1}$ is the natural inclusion of $\mathcal{M}_{k}(N)$ into $\mathcal{M}_{k}(Nd)$.

\begin{proposition}\label{proposition:bracket_operators_commute}
If $d,e$ are two positive integers and $\delta,\epsilon\in\{+,-\}$, then operators $[d]^{\delta}$ and $[e]^{\epsilon}$ commute.
\end{proposition}
\begin{proof}
By definition the operators $A_{d}$ and $A_{e}$ commute, so the proposition follows.
\end{proof}

We compute the action of $U_{p}$ and $[p]^{\pm}$ on $E_{k}$ explicitly. We adopt the convention that $\sigma_{k-1}(r)=0$ for any $r\in\mathbb{Q}\setminus\mathbb{Z}$.
\begin{lemma}
Let $k >2$ and $p$ be a prime number. We have equalities
\begin{align}
U_{p}([p]^{+}E_{k})&=[p]^{+}E_{k},\label{align:U_p_compose_p_plus}\\
U_{p}([p]^{-}E_{k})&=p^{k-1} [p]^{-}E_{k}.
\end{align}
\end{lemma}
\begin{proof}
Fix the integer $k>2$ and a prime $p$. We denote by $F$ the form $[p]^{+}E_{k}$, it lies in $\mathcal{M}_{k}(p)$. The $n$-th Fourier coefficient of $U_{p}F$ is as follows
\[a_{n}(U_{p}F)=a_{np}(F)=a_{np}(E_{k}-p^{k-1}A_{p}E_{k})=a_{np}(E_{k})-p^{k-1}a_{n}(E_{k}).\]
From the definition of the series $E_{k}$ we finally obtain
\[a_{n}(U_{p}F)=\sigma_{k-1}(np)-p^{k-1}\sigma_{k-1}(n).\]
On the other hand, the $n$-th Fourier coefficient of $F$ is equal to
\[\sigma_{k-1}(n)-p^{k-1}\sigma_{k-1}(n/p).\]
An application of the identity (\ref{align:sigma2}) shows that $U_{p}F=F$. 

\medskip\noindent
A similar reasoning combined with equation (\ref{align:sigma3}) proves the second statement of the lemma.
\end{proof}

For a square-free level $N$ we can now express the action of the Hecke algebra on a specific Eisenstein eigenform.
\begin{lemma}\label{lemma:eigenforms_Eisenstein}
Let $k>2$. Fix a positive integer $t$ and a non-negative integer $t\geq r$ and distinct prime numbers $p_{1},\ldots,p_{t}$. Let $N$ be a product of those primes. The form
\[E=[p_{1}]^{+}\circ\ldots\circ [p_{r}]^{+}\circ [p_{r+1}]^{-}\circ\ldots\circ [p_{t}]^{-} E_{k}\in\mathcal{E}_{k}(\Gamma_{0}(N))  \]
is an eigenform with respect to $\mathbb{T}_{N}$. Explicitly, the generators act as follows
\begin{alignat*}{3}
T_{n}E &=\sigma_{k-1}(n)E, &&\quad(n,N)=1\\
U_{p_{i}}E &=E, &&\quad 1\leq i\leq r\\
U_{p_{i}}E &=p_{i}^{k-1}E, &&\quad r+1\leq i\leq t
\end{alignat*}
\end{lemma}
\begin{proof}
Let $\ell$ be a prime number not dividing $N$. Equality (\ref{align:atkin_prop2}) and the definitions of $[p]^{+}$ and $[p]^{-}$ imply that operators $T_{\ell}$ and $[p_{i}]^{\pm}$ commute for any $i$ in the range $\{1,\ldots,t\}$ and for any choice of the sign $\pm$. It follows that
\[T_{\ell}E=[p_{1}]^{+}\circ\ldots\circ [p_{r}]^{+}\circ [p_{r+1}]^{-}\circ\ldots\circ [p_{t}]^{-} (T_{l}E_{k}).\]
Equality (\ref{equation:T_n_on_E_k}) implies that $T_{\ell}E=\sigma_{k-1}(\ell)E$. The operator $T_{\ell^{s}}$ for a fixed $s>1$ equals $P(T_{\ell})$ for a specific choice of $P\in\mathbb{Z}[x]$, so $T_{\ell^{s}}E=P(\sigma_{k-1}(\ell))E$. The polynomial $P$ is determined by the recurrence relation
\[T_{\ell^{s}}=T_{\ell}T_{\ell^{s-1}}-\ell^{s-1}T_{\ell^{s-2}}.\]
If we put $n=\ell^{s-1}$ in equation (\ref{align:sigma1}) the equation $P(\sigma_{k-1}(\ell))=\sigma_{k-1}(\ell^{s})$ follows, so $T_{\ell^{s}}E=\sigma_{k-1}(\ell^{s})E$. For a given $n$ coprime to $N$ the equation $T_{n}E =\sigma_{k-1}(n)E$ follows now from the definition of $T_{n}$ and the fact that $\sigma_{k-1}$ is a multiplicative function.

\medskip\noindent
Let $i$ be a fixed number in the set $\{1,\ldots,r\}$. Equation (\ref{align:atkin_prop3}) implies that $U_{p_{j}}\circ [p_{i}]^{+}=[p_{i}]^{+}\circ U_{p_{j}}$ and $U_{p_{j}}\circ [p_{i}]^{-}=[p_{i}]^{-}\circ U_{p_{j}}$ for any $j\neq i$. Proposition \ref{proposition:bracket_operators_commute} implies that the form $E$ can be written as
\[E=[p_{1}]^{+}\circ\ldots\circ [p_{i-1}]^{+}\circ[p_{i+1}]^{+}\circ\ldots\circ [p_{r}]^{+}\circ [p_{r+1}]^{-}\circ\ldots\circ [p_{t}]^{-}\circ [p_{i}]^{+} E_{k}.\]
and $U_{p_{i}}$ acts on $E$ in the following way
\[U_{p_{i}}E=[p_{1}]^{+}\circ\ldots\circ [p_{i-1}]^{+}\circ[p_{i+1}]^{+}\circ\ldots\circ [p_{r}]^{+}\circ [p_{r+1}]^{-}\circ\ldots\circ [p_{t}]^{-}\circ U_{p_{i}} [p_{i}]^{+} E_{k}.\]
Equation $U_{p_{i}}E=E$ is a direct consequence of (\ref{align:U_p_compose_p_plus}). For $i>r$ we proceed in a similar way to show $U_{p_{i}}E=p_{i}^{k-1}E$. The Hecke algebra $\mathbb{T}_{N}$ is generated by operators $T_{n}$ for $(n,N)=1$ and $U_{p_{i}}$ for $1\leq i\leq t$, so the above argument shows that $E$ is an eigenform with respect to $T_{N}$. 
\end{proof}

We now construct a basis of eigenforms for $k>2$ and $N$ square-free. If $N=N^{-}N^{+}$ is a decomposition into two possibly trivial factors, we define
\begin{equation}\label{equation:definition_E_N_N}
E_{N^{-},N^{+}}^{(k)}=[q_{1}]^{\epsilon_{1}}\circ \ldots [q_{t}]^{\epsilon_{t}}E_{k} 
\end{equation}
where $t$ is the number of prime factors of $N$ and $q_{1},\ldots,q_{t}$ are the prime factors of $N$. For $i$ in $\{1,\ldots,t\}$ we define
\[\epsilon_{i}=\left\{\begin{array}{cc}
  +, & \quad\textrm{if  }\  q_{i}|N^{+},\\
  -, & \quad\textrm{if  }\  q_{i}|N^{-}.
  \end{array}\right.
\]
For $N=1$ we have only one form $E_{1,1}^{(k)}=E_{k}$. We will often drop the upper index in $E_{N^{-},N^{+}}^{(k)}$ and write $E_{N^{-},N^{+}}$ if it is clear from the context that the weight equals $k$.
\begin{theorem}\label{theorem:basis_weight_gt_2}
Let $k>2$ and let $N$ be a square-free integer. The set
\[B:=\{E_{N^{-},N^{+}}^{(k)}:N=N^{-}N^{+}\}\]
forms a $\mathbb{C}$-basis of the vector space $\mathcal{E}_{k}(N)$. Each element of this basis is an eigenform with respect to the Hecke algebra $\mathbb{T}_{N}$. The cardinality of the basis is $2^{t}$ where $t$ is the number of prime factors of $N$.
\end{theorem}
\begin{proof}
Forms from the set $B$ are linearly independent because they have different sets of eigenvalues with respect to the Hecke algebra $\mathbb{T}_{N}$, cf. Lemma \ref{lemma:eigenforms_Eisenstein}. Let $d(N)$ denote the number of divisors of $N$. We can choose $N^{-}$ from $d(N)$ possible divisors of $N$, the factor $N^{+}$ is determined by this choice. Hence the cardinality of $B$ equals $d(N)=2^{t}$. But from \cite[Theorem 3.5.1]{Diamond_Shurman} and \cite[p.103]{Diamond_Shurman} we know that the dimension of the space $\mathcal{E}_{k}(N)$ equals $2^{t}$, so $B$ is a basis of this space.
\end{proof}

\begin{corollary}\label{corollary:initial_coefficient_weight_gt_2}
Let $k>2$ and let $N$ be a square-free integer with prime factors $p_{1},\ldots,p_{t}$. Choose a form $E_{N^{-},N^{+}}\in\mathcal{E}_{k}(N)$ which is an eigenform. Let $a_{0}(E_{N^{-},N^{+}})$ denote the initial coefficient of the $q$-expansion of $E$ at infinity. Then
\begin{alignat*}{3}
a_{0}(E_{N^{-},N^{+}}) &=-\frac{B_{k}}{2k}\prod_{i=1}^{t}(1-p_{i}^{k-1}),  &&\quad\textrm{if } N^{-}=1\\
a_{0}(E_{N^{-},N^{+}}) &=0,  &&\quad\textrm{if } N^{-}>1
\end{alignat*}
\end{corollary}
\begin{proof}
Observe that for any form $f$ and prime $p$ we have $a_{0}([p]^{-}f)=0$. The operators $[\cdot]^{-}$ and $[\cdot]^{+}$ commute, so when $N^{-}>1$ we can write $E_{N^{-},N^{+}}$ as $[p]^{-}f$ where $p$ is prime and $f$ is a form in $\mathcal{E}_{k}(N/p)$, hence $a_{0}(E_{N^{-},N^{+}})=0$. Now for any form $f$ and prime $p$ we obtain
\begin{equation}\label{equation:a_0_for_p_plus}
a_{0}([p]^{+}f)=a_{0}(f)(1-p^{k-1}). 
\end{equation}
So if $N^{-}=1$ we obtain 
\[a_{0}(E_{N^{-},N^{+}}) =-\frac{B_{k}}{2k}\prod_{i=1}^{t}(1-p_{i}^{k-1})\]
if we apply successively equation (\ref{equation:a_0_for_p_plus}) to each factor of $N$. Finally we recall that $a_{0}(E_{k})=-\frac{B_{k}}{2k}$.
\end{proof}
In weight $k=2$ the series $E_{2}$ does not define a modular form in $\mathcal{M}_{2}(1)$, so in order to find the basis of eigenforms in $\mathcal{E}_{2}(N)$ we need to modify the argument above. It is well-known that for a prime $p$ the form $[p]^{+}E_{2}$ is a modular form in $\mathcal{E}_{2}(p)$.
\begin{lemma}
Let $p$ be a prime number. The form $[p]^{+}E_{2}\in\mathcal{E}_{2}(p)$ is an eigenform with respect to the Hecke algebra $\mathbb{T}_{p}$. The Fourier coefficient $a_{1}([p]^{+}E_{2})$ is 1 and for a prime $q\neq p$ the $q$-th Fourier coefficient of $[p]^{+}E_{2}$ is $q+1$. The following identities hold
\begin{align*}
U_{p}([p]^{+}E_{2})&=[p]^{+}E_{2},\\
T_{n}([p]^{+}E_{2})&=a_{n}(E_{2})[p]^{+}E_{2},\quad\textrm{for }(n,Np)=1.
\end{align*}
\end{lemma}
\begin{proof}
Let $\ell\neq p$ be a prime number. For a fixed integer $n$ we obtain
\[a_{n}(T_{\ell}([p]^{+}E_{2}))=\sigma_{1}(n\ell)-p\sigma_{1}(n\ell/p)+\ell\sigma_{1}(n/\ell)-\ell p\sigma_{1}(n/(\ell p)).\]
On the other hand we know that
\[(1+\ell)a_{n}([p^{+}]E_{2})=(1+\ell)(\sigma_{1}(n)-p\sigma_{1}(n/p)).\]
We now apply equation (\ref{align:sigma1}) to obtain
\[a_{n}(T_{\ell}([p]^{+}E_{2}))=(1+\ell)a_{n}([p^{+}]E_{2}).\]
It is easy to see that $a_{n}(U_{p}[p]^{+}E_{2})=a_{np}([p]^{+}E_{2})$ and
\[a_{np}([p]^{+}E_{2})=\sigma_{1}(np)-p\sigma_{1}(n)=\sigma_{1}(n)-p\sigma_{1}(n/p)=a_{n}([p]^{+}E_{2}).\]
The third equation is a consequence of (\ref{align:sigma2}). Hence, the form $[p]^{+}E_{2}$ is an eigenform with respect to $U_{p}$ and any $T_{\ell}$ for $\ell\neq p$, so it is an eigenform with respect to $\mathbb{T}_{p}$. 
The second equation from the statement of the lemma follows from the definition of $T_{n}$ and from the multiplicativity of $\sigma_{1}$. 
From the definition we also obtain that $a_{1}([p]^{+}E_{2})=a_{1}(E_{2})=\sigma_{1}(1)=1$ and also $a_{q}([p]^{+}E_{2})=a_{q}(E_{2})=\sigma_{1}(q)=1+q$ for any prime $q\neq p$.
\end{proof}

\begin{lemma}\label{lemma:basis_for_weight_2}
Let $N>1$ be a square-free integer. Suppose $f\in\mathcal{E}_{2}(N)$ is an eigenform with respect to $\mathbb{T}_{N}$ such that $a_{1}(f)=1$ and $a_{q}(f)=1+q$ for any prime $q\nmid N$. For a fixed prime $p\nmid N$ the forms $[p]^{+}f$ and $[p]^{-}f\in\mathcal{M}_{2}(Np)$ are eigenforms with respect to $\mathbb{T}_{Np}$. The following identities hold
\begin{align*}
U_{p}([p]^{+}f)&=[p]^{+}f,\\
U_{p}([p]^{-}f)&=p[p]^{-}f,\\
T_{n}([p]^{+}f)&=a_{n}(f)[p]^{+}f,\quad\textrm{for }(n,Np)=1,\\
T_{n}([p]^{-}f)&=a_{n}(f)[p]^{-}f,\quad\textrm{for }(n,Np)=1.
\end{align*}
Moreover $a_{1}([p]^{\pm}f)=1$ and $a_{q}([p]^{\pm}f)=1+q$ for any prime $q\nmid Np$.
\end{lemma}
\begin{proof}
Let $\ell$ be a prime not dividing $Np$. Formula (\ref{align:atkin_prop2}) implies that $$(T_{\ell}\circ[p]^{+})f=([p]^{+}\circ T_{\ell})f.$$ The form $f$ is normalized so $T_{\ell}f=a_{\ell}(f)f$ and it follows that $$T_{\ell}([p]^{+}f)=a_{\ell}(f)[p]^{+}f.$$ In a similar way we show that $T_{\ell}([p]^{-}f)=a_{\ell}(f)[p]^{-}f$. From the multiplicativity of $\sigma_{1}$, definition of $E_{2}$ and of $T_{n}$ for $(n,Np)=1$ we obtain the third and fourth equation from the statement of this lemma.

\medskip\noindent
The equality $U_{p}([p]^{-}f)=p\cdot [p]^{-}f$ is equivalent to 
\begin{equation}\label{equation:equal_coeff_U_p}
a_{np}(f)-a_{n}(f)=p(a_{n}(f)-a_{n/p}(f)).
\end{equation}
If $f$ is a normalized eigenform, then $a_{np}(f)=a_{n}(f)a_{p}(f)$ for $p\nmid n$. Since $p\nmid N$ we obtain $a_{p}(f)=1+p$ and equation (\ref{equation:equal_coeff_U_p}) holds. In the case $n=n'p^{\alpha}$ for $\alpha>0$ equation (\ref{equation:equal_coeff_U_p}) is equivalent to
\[a_{p^{\alpha+1}}(f)=(p+1)a_{p^{\alpha}}(f)-pa_{p^{\alpha-1}}(f).\]
This clearly holds because $f$ is an eigenform for $\mathbb{T}_{N}$ and we therefore have the recurrence relation $T_{p^{\alpha+1}}=T_{p}T_{p^{\alpha}}-pT_{p^{\alpha-1}}$ and $a_{p}(f)=1+p$. It is possible to show that $U_{p}([p]^{+}f)=[p]^{+}f$ in a similar fashion. Finally, the equalities $a_{1}([p]^{\pm}f)=1$ and $a_{q}([p]^{\pm}f)=1+q$ follow from the assumptions made on $f$ and from definitions of $[p]^{\pm}$.
\end{proof}
For $k=2$ we can adopt the notation $E_{N^{-},N^{+}}^{(k)}$ from (\ref{equation:definition_E_N_N}) with one small exception: we require that $N^{+}>1$.
\begin{theorem}\label{theorem:basis_for_weight_2}
Let $N>1$ be a square-free integer. The set
\[B:=\{E_{N^{-},N^{+}}^{(2)}:N=N^{-}N^{+}, N^{+}>1\}\]
forms a $\mathbb{C}$-basis of the vector space $\mathcal{E}_{2}(N)$. Each element of this basis is an eigenform with respect to the Hecke algebra $\mathbb{T}_{N}$. The cardinality of the basis is $2^{t}-1$ where $t$ is the number of prime factors of $N$.
\end{theorem}
\begin{proof}
We virtually repeat the proof of Theorem \ref{theorem:basis_weight_gt_2} replacing Lemma \ref{lemma:eigenforms_Eisenstein} with Lemma \ref{lemma:basis_for_weight_2}. The set $B$ has one less element in this case and we compare it with \cite[Theorem 3.5.1]{Diamond_Shurman} to prove that $B$ is a basis of the $\mathcal{E}_{2}(N)$. 
\end{proof}
\begin{remark}
Theorem \ref{theorem:basis_for_weight_2} is proved in \cite[\S 2]{Yoo_multiplicity} in another way and the proof requires additional tools which are not necessary in our proof.
\end{remark}

\begin{corollary}
Let $N$ be a square-free integer with prime factors $p_{1},\ldots,p_{t}$. Choose a form $E=E_{N^{-},N^{+}}\in\mathcal{E}_{2}(N)$ which is an eigenform. Then
\begin{alignat*}{3}
a_{0}(E_{N^{-},N^{+}}) &=-\frac{B_{2}}{4}\prod_{i=1}^{t}(1-p_{i}),  &&\quad\textrm{if } N^{-}=1\\
a_{0}(E_{N^{-},N^{+}}) &=0,  &&\quad\textrm{if } N^{-}>1
\end{alignat*}
\end{corollary}
\begin{proof}
If $N^{-}>1$, then $E_{N^{-},N^{+}}$ is of the form $[p]^{-}h$ for some $h\in\mathcal{E}_{2}(N/p)$ and a prime $p\mid N$, hence $a_{0}(E_{N^{-},N^{+}})=0$. For the case $N^{-}=1$ we simply use that $a_{0}([p]^{+}h)=a_{0}(h)(1-p)$.
\end{proof}

\section{Upper bound of congruences}\label{sec:Upper_bound_of_congruences}
In this section we discuss a general upper bound for the exponent of congruences between cuspidal eigenforms and eigenforms in the Eisenstein subspace for square-free levels $N$ and even weights $k\geq 2$. The theorems proved here generalize the results obtained previously in \cite{Naskrecki_Springer}.

\begin{lemma}[{\cite[Theorem 3]{Atkin_Lehner}}]\label{lemma:Atkin_Lehner_special_coeff}
Let $N$ be a square-free integer and let $k\geq 2$ be an even integer. If $f\in\mathcal{S}_{k}(N)^{\textrm{new}}$ is a newform, then for any $p\mid N$ we have
\[a_{p}(f)=-\lambda_{p}p^{k/2-1},\]
where $\lambda_{p}\in\{\pm 1\}$.
\end{lemma}

Let $K$ be a number field and $\mathcal{O}_{K}$ its ring of integers. For an element $\alpha\in\mathcal{O}_{K}$ and a maximal ideal $\lambda\subset\mathcal{O}_{K}$ let $\mathrm{ord}_{\lambda}(\alpha)$ denote the integer that satisfies the condition
\[n\leq \mathrm{ord}_{\lambda}(\alpha)\quad\Longleftrightarrow\quad\lambda^{n}\mid \alpha\mathcal{O}_{K}.\]
We can naturally extend $\mathrm{ord}_{\lambda}$ to a function on $K^{\times}$. For a prime $\ell\in\mathbb{Z}$ let $v_{\ell}$ denote the standard $\ell$-adic valuation on $\mathbb{Q}^{\times}$. For any $a\in\mathbb{Q}^{\times}$ we have $\mathrm{ord}_{\lambda}(a)=\mathrm{ord}_{\lambda}(\ell)v_{\ell}(a)$ where $\ell$ is the field characteristic of $\mathcal{O}_{K}/\lambda$.

\medskip\noindent
Let $K_{f}$ denote the field of coefficients of the newform $f\in\mathcal{S}_{k}(N)^{\mathrm{new}}$ and by $\mathcal{O}_{f}$ its ring of integers.

\medskip\noindent
Let $f,g\in\mathcal{M}_{k}(N)$ be two eigenforms and $K$ be a field that contains the composite of $K_{f}$ and $K_{g}$. We say that $f$ and $g$ are congruent modulo a power $\lambda^r$ of a maximal ideal $\lambda\in\mathcal{O}_{K}$ if and only if
\[\modulo{a_{n}(f)}{a_{n}(E)}{\lambda^{r}}\]
for all $n\geq 0$, where $\{a_{n}(f)\}$ and $\{a_{n}(g)\}$ are Fourier coefficient of the $q$-expansion at infinity of $f$ and $g$, respectively. 

\begin{theorem}\label{corollary:bound_for_a_0_nonzero}
Let $p_{1},\ldots,p_{t}$ be different prime factors of $N$ and let $k\geq 2$. Suppose that $f\in\mathcal{S}_{k}(N)^{\mathrm{new}}$ is a newform which is congruent to the Eisenstein eigenform $E=[p_{1}]^{+}\circ\ldots\circ[p_{t}]^{+}E_{k}\in\mathcal{E}_{k}(N)$ modulo a power $r>0$ of a maximal ideal $\lambda\subset\mathcal{O}_{f}$. If $\ell$ is the residual characteristic of $\lambda$, we obtain the bound
\[r\leq ord_{\lambda}(\ell)\cdot v_{\ell}\left(-\frac{B_{k}}{2k} \prod_{i=1}^{t}(1-p_{i})\right). \]
\end{theorem}
\begin{proof}
Let $p\mid N$ be a prime. From Lemma \ref{lemma:Atkin_Lehner_special_coeff} we know that $a_{p}(f)=-\lambda_{p}p^{k/2-1}$. On the other hand $a_{p}(E)=a_{1}(U_{p}E)$ and from Lemma \ref{lemma:eigenforms_Eisenstein} it follows that $a_{p}(E)=1$. The congruence
\[\modulo{a_{p}(f)}{a_{p}(E)}{\lambda^{r}}\]
implies that $\modulo{-\lambda_{p}p^{k/2-1}}{1}{\lambda^{r}}$ and by squaring both sides we obtain an equation
\begin{equation}\label{eq:kongr_a_p_dowod}
\modulo{1-p^{k-2}}{0}{\lambda^{r}}.
\end{equation}
Since $f$ is a cusp form, $a_{0}(E)\equiv a_{0}(f)=0\textrm{ }(\textrm{mod }\lambda^{r})$ holds and by Corollary \ref{corollary:initial_coefficient_weight_gt_2} we obtain
\[\modulo{-\frac{B_{k}}{2k}\prod\limits_{i=1}^{t}(1-p_{i}^{k-1})}{0}{\lambda^{r}}.\]
We observe that $1-p_{i}^{k-1}=(1-p_{i}^{k-2})+p_{i}^{k-2}(1-p_{i})$. The equation (\ref{eq:kongr_a_p_dowod}) holds for each $p_{i}$ under the assumption $\ell\nmid N$. It follows that
\[\modulo{-\frac{B_{k}}{2k}\prod\limits_{i=1}^{t}(1-p_{i})}{0}{\lambda^{r}}.\]
For $k\geq 2$ we have the inequality $ord_{\lambda}(1-p_{i}^{k-1})\geq ord_{\lambda}(1-p_{i})$ for each $i$, hence
\[r\leq ord_{\lambda}\left(-\frac{B_{k}}{2k}\prod\limits_{i=1}^{t}(1-p_{i})\right).\]
\end{proof}

\begin{corollary}\label{corollary:bound_when_a_0_is_0}
Let $p_{1},\ldots,p_{t}$ be different prime factors of $N$ and let $k\geq 2$. Suppose $f\in\mathcal{S}_{k}(N)^{\mathrm{new}}$ is a newform which is congruent to the Eisenstein eigenform $E=[p_{1}]^{\epsilon_{1}}\circ\ldots\circ[p_{t}]^{\epsilon_{t}}E_{k}\in\mathcal{E}_{k}(N)$ modulo a power $\lambda^r$ of a maximal ideal $\lambda\subset\mathcal{O}_{f}$. If we assume that $a_{0}(E)=0$ and $p_{i}\notin\lambda$ for every $\epsilon_{i}=-$, then we have the following bound for the congruence exponent
\[r\leq\min\{\min_{i,\epsilon_{i}=+}ord_{\lambda}(1-p_{i}^{k-2}),\min_{i,\epsilon_{i}=-}ord_{\lambda}(1-p_{i}^{k})\}.\]
Moreover, for every $i$ such that $\epsilon_{i}=+$ we have $p_{i}\notin\lambda$.
\end{corollary}
\begin{proof}
We apply Lemma \ref{lemma:Atkin_Lehner_special_coeff} to the congruence $\modulo{a_{p_{i}}(f)}{a_{p_{i}}(E)}{\lambda^{r}}$. After squaring both sides we obtain the condition
\begin{equation}\label{equation:congr_for_p_i}
p_{i}^{k-2}\equiv \left\{\begin{array}{ll}
1, & \textrm{for }\epsilon_{i}=+,\\
p_{i}^{2(k-1)}, & \textrm{for }\epsilon_{i}=-.
\end{array}\right.
\end{equation}
The exponent $r$ is less than or equal to $ord_{\lambda}(1-p_{i}^{k-2})$ when $\epsilon_{i}=+$. Also $r$ is at most equal to $ord_{\lambda}(1-p_{i}^{k})$ when $\epsilon_{i}=-$, because $p_{i}\notin \lambda$ by assumption. For each $i$ such that $\epsilon_{i}=+$ it follows from the congruence (\ref{equation:congr_for_p_i}) that $1-p_{i}^{k-2}\in\lambda^{r}$. So $1-p_{i}^{k-2}\in\lambda$ and then $p_{i}\notin\lambda$.
\end{proof}

\section{Rational congruences}\label{sec:Rational_cong}
We have proved in \cite[\S 5.8]{Naskrecki_Springer} that for a prime $N$ and a newform $f\in\mathcal{S}_{2}(\Gamma_{0}(N))^{\mathrm{new}}$ with rational coefficients there exists a system of congruences
\begin{equation}\label{equation:rational_congruences}
\modulo{a_{n}(f)}{a_{n}(E)}{\ell^{r}}
\end{equation}
for all $n\geq 0$, $E=[N]^{+}E_{2}$ and a rational prime $\ell$ only for triples $(\ell,r,N)\in\{(3,1,19),(3,1,37),(5,1,11),(2,1,17)\}$ (only finitely many systems) and also for $(\ell,r,N)\in\{(2,1,u^2+64): 2\nmid u\}$ (conjecturally infinitely many triples).
\begin{lemma}
Let $f$ be a newform $f\in\mathcal{S}_{2}(\Gamma_{0}(N))^{\mathrm{new}}$ with rational coefficients and $N$ a square-free number. Suppose we have an eigenform $E\in\mathcal{E}_{2}(N)$ and the congruence (\ref{equation:rational_congruences}) holds for all $n\geq 0$, then
\[(\ell,r)\in\{(2,1),(2,2),(2,3),(3,1),(3,2),(5,1),(7,1)\}.\]
\end{lemma}
\begin{proof}
We know that the Fourier coefficients of $f$ at infinity are integers
\cite[Theorem 6.5.1]{Diamond_Shurman} and for every prime $q\nmid N$
\begin{equation}
\modulo{a_{q}(f)}{1+q}{\ell^{r}}.
\end{equation}
There exists an elliptic curve $\mathcal{E}$ over $\mathbb{Q}$ of conductor $N$ such that for a prime $q$ of good reduction for $\mathcal{E}$, $a_{q}(f)=q+1-|\mathcal{E}(\mathbb{F}_{q})|$, \cite[Chapter II, \S 2.6]{Cremona_Modular}. By the theorem of Katz there exists a $\mathbb{Q}$-isogenous curve $\mathcal{E}'$ such that $\mathcal{E}'(\mathbb{Q})$ contains an $\ell^{r}$-torsion point. By the theorem of Mazur \cite{Mazur_modular} it follows that $\ell^{r}\in\{2,3,4,5,7,8,9\}$.
\end{proof}

Elliptic curves with conductor $N$ a product of two primes were partially classified in \cite{Sadek_torsion}. This result allows us to discard the congruences with $\ell^{r}\in\{8,9\}$.
\begin{theorem}\label{thm:congruences_over_Q}
Let $p$,$q$ be two different primes. Suppose that $f\in\mathcal{S}_{2}(pq)^{\mathrm{new}}$ is a newform with rational coefficients and let $E$ be an eigenform in $\mathcal{E}_{2}(pq)$. Let $\ell$ be a prime number and $r>0$ an integer such that the congruence (\ref{equation:rational_congruences}) holds for all $n\geq 0$. Then one of two conditions holds
\begin{itemize}
\item[(1)] $\ell^{r}\in\{2,3,4,5\}$ or
\item[(2)] $\ell^{r}=7$ and $E=[13]^{-}[2]^{+}E_{2}$.
\end{itemize}
\end{theorem}
\begin{proof}
	Let $N=pq$ be odd. Then $\modulo{a_2(f)}{3}{\ell^r}$. From the Hasse-Weil bound it follows that $|a_{2}(f)|\leq 2\sqrt{2}< 3$. Hence $\ell^{r} \mid (3-a_2(f))< 6$ so we conclude (1).
	When $N=pq$ is even and $N=6$ then the set of cusp forms is empty. So we can assume that $p=2$ and $q>3$. Then the inequality $|a_{3}(f)|\leq 2\sqrt{3}<4$ and the congruence $\modulo{a_{3}(f)}{4}{\ell^{r}}$ holds, hence $\ell^{r}\mid (4-a_{3}(f)) <8$. For $\ell^{r}=7$ by \cite[Theorem 3.6]{Sadek_torsion} it follows that $N=26$. We compute that the space $\mathcal{S}_{2}(26)^{\mathrm{new}}$ is of dimension $2$ and spanned by the forms $f_{1},f_{2}$ with the following Fourier expansions
\begin{align*}
f_{1}&=q - q^2 + q^3 + q^4 - 3q^5 - q^6 - q^7 - q^8 - 2q^9 + 3q^{10} + 6q^{11}+\ldots\\
f_{2}&=q + q^2 - 3q^3 + q^4 - q^5 - 3q^6 + q^7 + q^8 + 6q^9 - q^{10} - 2q^{11}+\ldots
\end{align*}
The space $\mathcal{E}_{2}(26)$ has a basis consisting of three eigenforms $$[2]^{-}[13]^{+}E_{2},\quad [13]^{-}[2]^{+}E_{2}, \quad [2]^{+}[13]^{+}E_{2}.$$ Lemma \ref{lemma:basis_for_weight_2} implies that
\begin{align*}
a_{2}([2]^{-}[13]^{+}E_{2})&=2,\\
a_{2}([13]^{-}[2]^{+}E_{2})&=1,\\
a_{2}([2]^{+}[13]^{+}E_{2})&=1.
\end{align*}
The Sturm bound is $7$ by Theorem \ref{theorem:generalized_sturm_theorem}, so we only have to compare $7$ initial coefficients to verify the desired congruence. By a direct computation we see that $f_{2}$ is congruent to $[13]^{-}[2]^{+}E_{2}$ modulo $7$. The form $f_{1}$ is not congruent to any of the given Eisenstein eigenforms modulo $7$.
\end{proof}

\begin{remark}
If $N$ has more than two prime factors we can find examples of congruences where $\ell^{r}\in\{8,9\}$. In Tables \ref{table:congr_2_3} and \ref{table:congr_3_2} we present such examples. The index notation $f_i$ of the modular forms is described in Section \ref{sec:descr_of_data}.

{\def\arraystretch{1.2}
\begin{table}[ht]
\begin{center}
\setlength{\tabcolsep}{8pt}
\csvreader[tabular=c|c|c|c|c,
    table head= & $N$ & $N^-$ & $N^+$ & form\\\hline,
    late after line=\\  ]%
{r_2_l_3_data.csv}{n=\n,nm=\nm,np=\np,fnumber=\fnumber}%
{\thecsvrow & \n & \nm & \np & $f_{\fnumber}$ }
\end{center}
\caption{$\modulo{a_{n}(f_{i})}{a_{n}(E_{N^{-},N^{+}})}{2^3},\,\, n\geq 0,\,\, f_{i}\in\mathcal{S}_{2}(\Gamma_{0}(N))^{\textrm{new}}$}\label{table:congr_2_3}
\end{table}
}

{\def\arraystretch{1.2}
\begin{table}[ht]
\begin{center}
\setlength{\tabcolsep}{8pt}
\csvreader[tabular=c|c|c|c|c,
    table head= & $N$ & $N^-$ & $N^+$ & form\\\hline,
    late after line=\\ ]%
{r_3_l_2_data.csv}{n=\n,nm=\nm,np=\np,fnumber=\fnumber}%
{\thecsvrow & \n & \nm & \np & $f_{\fnumber}$ }
\end{center}
\caption{$\modulo{a_{n}(f_{i})}{a_{n}(E_{N^{-},N^{+}})}{3^2},\,\, n\geq 0,\,\, f_{i}\in\mathcal{S}_{2}(\Gamma_{0}(N))^{\textrm{new}}$}\label{table:congr_3_2}
\end{table}
}

\end{remark}

\section{Algorithmic search for congruences}\label{sec:Algorithms}
Our main goal in this section is to describe an effective algorithm that allows one to find congruences between cuspidal eigenforms and Eisenstein series for a large class of square-free conductors. Our approach follows \cite{Sturm} and an adaptation of Sturm's algorithm given in \cite{Kiming_mod_powers}.

\begin{theorem}\label{theorem:generalized_sturm_theorem}
Let $p_{1},\ldots,p_{t}$ be different prime numbers and $k\geq 2$. Let $N=p_{1}\cdot\ldots\cdot p_{t}$ and $f$ be a newform in $\mathcal{S}_{k}(N)^{\mathrm{new}}$. We fix a natural number $r$ and a maximal ideal $\lambda$ in $\mathcal{O}_{f}$. Let $E$ be an eigenform in $\mathcal{E}_{k}(N)$. If the congruence
\begin{equation}\label{equation:congruence_for_E_f}
a_{n}(f)\equiv a_{n}(E)\textrm{ mod }\lambda^{r}
\end{equation}
holds for all $n\leq k(\prod_{i} (p_{i}+1))/12$, then it holds for all $n\geq 0$.
\end{theorem}
\begin{proof}
This is a simple adaptation of \cite[Proposition 1]{Kiming_mod_powers}.
\end{proof}

In our algorithm it will be sufficient to check the condition (\ref{equation:congruence_for_E_f}) for indices $n$ that are prime numbers below the Sturm bound $B:=k(\prod_{i} (p_{i}+1))/12$.
\begin{corollary}\label{corollary:Sturm_on_primes}
With the assumptions as in Theorem \ref{theorem:generalized_sturm_theorem} suppose that for primes $n\leq k(\prod_{i} (p_{i}+1))/12$ the congruence (\ref{equation:congruence_for_E_f}) holds, then the congruence (\ref{equation:congruence_for_E_f}) holds for all natural numbers $n\geq 0$.
\end{corollary}
\begin{proof}
	This follows immediately from Theorem \ref{theorem:generalized_sturm_theorem} since $f$ and $E$ are normalized eigenforms.
\end{proof}

\begin{lemma}\label{lem:bound_primes}
Let $N$ be a square-free integer which is a product of prime numbers $p_{1},\ldots, p_{t}$ and $k\geq 2$ be an integer. Let $\{\epsilon_{i}\}_{i=1,\ldots,t}$ be a collection of symbols $\epsilon_{i}\in\{+,-\}$. Let $f\in\mathcal{S}_{k}(N)^{\mathrm{new}}$ be a newform and $E\in\mathcal{E}_{k}(N)$ an eigenform $E=[p_{1}]^{\epsilon_{1}}\circ\ldots\circ[p_{t}]^{\epsilon_{t}}E_{k}\in\mathcal{E}_{k}(N)$. Suppose that there exists a prime ideal $\lambda$ in $\mathcal{O}_{f}$ and a positive integer $r$ for which the congruence $\modulo{a_{n}(f)}{a_{n}(E)}{\lambda^{r}}$ holds for all integers $n$. Let $\ell$ denote the characteristic of the field $\mathcal{O}_{f}/\lambda$. One of the following conditions holds
\begin{itemize}
	\item[(1)] If $k\geq 2$ and $\epsilon_{1}=\ldots=\epsilon_{t}=+$, then $\ell\mid -\frac{B_{k}}{2k}\prod_{i}(1-p_{i})$.
	\item[(2)]  If $k=2$ and $\epsilon_{i}=-$ for some $i$, then $\ell\mid \textrm{GCD}(\{1-p_j^2:\epsilon_{j}=-\})$
	\item[(3)]  If $k>2$ and $\epsilon_{i}=-$ for some $i$, then 
	$\ell\mid \textrm{GCD}(\{1-p_j^k:\epsilon_{j}=-\}\cup \{1-p_j^{k-2}:\epsilon_{j}=+\})$
\end{itemize}
\end{lemma}
\begin{proof}
The lemma follows from the Theorem \ref{corollary:bound_for_a_0_nonzero} and Corollary \ref{corollary:bound_when_a_0_is_0}.
\end{proof}

\subsection{Algorithm}\label{sec:algorithm}
\textit{Description of the algorithm:} For a fixed integer $k\geq 2$, a square-free integer $N$, a prime number $\ell$ and a fixed eigenform $E\in\mathcal{E}_{k}(N)$ the algorithm checks for which newforms $f\in\mathcal{S}_{k}(N)$ there is a congruence between $f$ and $E$ modulo $\lambda^{r}$ where the characteristic of the ideal $\lambda$ is $\ell$ and $r>0$ is the maximal possible.

\medskip\noindent
\textbf{Input:} an even number $k\geq 2$, a square-free integer $N>1$,  a prime number $\ell$ and an eigenform $E\in\mathcal{E}_{k}(N)$

\medskip\noindent
\textbf{Steps of the algorithm:}
\begin{itemize}
\item[(1)] Check whether $a_{0}(E)$ is $0$. If yes, then proceed to Step 2. If no, then check if $v_{\ell}(a_{0}(E))>0$. If yes, then go to Step 2. If no, then terminate the algorithm.
\item[(2)] Compute subsets $C_{i}$ of newforms in $\mathcal{S}_{k}(N)$ such that each two element in $C_{i}$ are Galois conjugate
\item[(3)] For each set $C_{i}$ pick one representative and create a set $F_{N,k}$ of those representatives for all $i$.
\item[(4)] Compute the Sturm bound $B=(k/12)[\mathrm{SL}_{2}(\mathbb{Z}):\Gamma_{0}(N)]$.
\item[(5)] For each form $f\in F_{N,k}$ compute the coefficient field $K_{f}$.
\item[(6)] For each $f\in F_{N,k}$ create a set $S_{\ell,f}$ that is made of prime ideals that appear in the factorization of $\ell\mathcal{O}_{f}$.
\item[(7)] For each element $f\in F_{N,k}$ and $\lambda\in S_{\ell,f}$ compute the number
\[r_{\lambda}=\min\left\{ord_{\lambda}\left(a_{q}(f)-a_{q}(E)  \right)\mid q\leq B\right\}.\]
The minimum runs over prime numbers $q$. If $r_{\lambda}>0$, then return a triple $(f,\lambda,r_{\lambda})$. 
\end{itemize}

\medskip\noindent
\textbf{Output:} Set of triples $(f,\lambda,r)$ such that
\[\modulo{a_{n}(f)}{a_{n}(E)}{\lambda^{r}}\]
for all $n\geq 0$ and if for some $s>0$ we have
\[\modulo{a_{n}(f)}{a_{n}(E)}{\lambda^{s}}\]
for all $n\geq 0$, then $s\leq r$. Remark: it might happen that the list will be empty.

\medskip\noindent
\textbf{Validity of the algorithm:} In Step 1 we check if the congruence (\ref{equation:congruence_for_E_f}) is possible. Step 2 amounts to a finite number of computational steps for a fixed level $N$ and weight $k$ by using, for instance, modular symbols. Moreover we can represent each newform by a finite number of bits (e.g. by using the modular symbols representation). The number $r_{\lambda}$ in Step 6 satisfies the output condition because of Corollary \ref{corollary:Sturm_on_primes}. Since $N$ is square-free the constant $B$ is equal to the constant from Corollary \ref{corollary:Sturm_on_primes}.

\section{Numerical data}\label{sec:Numerical_data}
In this section we present the computational data that was gathered while running Algorithm \ref{sec:algorithm} under the restrictions of Lemma \ref{lem:bound_primes}. We performed a check that includes weights $k$ between $2$ and $24$ and square-free levels $N$ up to $4559$. More precise bounds are presented in Table \ref{table:range}. Our main computational resource was the cluster Gauss at the University of Luxembourg maintained by Prof. Gabor Wiese. This computer has 20 CPU units of type Inter(R) Xeon(R) CPU E7-4850 @ 2.00 GHz and approximately 200 GB of RAM memory. We used the computer algebra package MAGMA \cite{MAGMA} and the set of instructions MONTES \cite{Montes} which greatly enhances the efficiency of computations performed on number fields with large discriminants. However, it took about 4 weeks of the computational time under full CPU load of the Gauss cluster (around 13440 CPU hours) to finish the calculations.

{\renewcommand{\arraystretch}{1.3}
\begin{table}[ht]
\begin{center}
\setlength{\tabcolsep}{6pt}
\begin{tabular}{c|c|c|c|c|c|c|c|c|c|c|c|c}
$k$     & 2 & 4 & 6 & 8 & 10 & 12 & 14 & 16 & 18 & 20 & 22 & 24\\
\hline
$N\leq$ & 4559 & 922 & 302 & 202 & 193 & 102 & 94 & 94 & 94 & 94 & 94 & 94
\end{tabular}
\end{center}
\caption{Weight $k$ and corresponding maximal level $N$.}\label{table:range}
\end{table}}

\subsection{Description of data in the tables}\label{sec:descr_of_data}
Let $f_{i}$ be a newform in $\mathcal{S}_{k}(N)^{\mathrm{new}}$ where $k\geq 2$ and $N$ is square-free. The index $i$ is associated to the particular form by the algorithm presented in \cite[Chapter IV]{Cremona_Modular}, described in details in MAGMA manual \footnote{\url{http://magma.maths.usyd.edu.au/magma/handbook/text/1545}}.The number $d$ will denote the degree of the extension $K_{f}$ over $\mathbb{Q}$. Let $\lambda\subset\mathcal{O}_{f_{i}}$ be a prime ideal with residue characteristic $\ell$. Let $e$ denote the ramification degree $\mathrm{ord}_{\lambda}(\ell)$ and $f$ the degree of the residue field extension $[\mathcal{O}_{f_{i}}/\lambda:\mathbb{F}_{\ell}]$. We consider the Eisenstein eigenform $E_{N^{-},N^{+}}\in\mathcal{E}_{k}(N)$ with $N=N^{-}N^{+}$ such that
\begin{equation}\label{equation:explicit_congruence_general}
\modulo{a_{n}(f_{i})}{a_{n}(E_{N^{-},N^{+}})}{\lambda^{r}}
\end{equation}
for all $n\geq 0$. Assume that the positive integer $r$ is maximal, i.e. there is no congruence of type \eqref{equation:explicit_congruence_general} with ideal exponent $r'$ greater than $r$. The number $m$ will denote the maximum over $s$ which satisfy simultaneous congruences
\[\modulo{a_{p_{j}}(f_{i})}{a_{p_{j}}(E_{N^{-},N^{+}})}{\ell^{s}},\quad 1\leq j\leq t, \]
\[\modulo{a_{0}(f_{i})}{a_{0}(E_{N^{-},N^{+}})}{\ell^{s}}.\]
Observe that $m$ depends on the choice of $N$, $N^{+}$, $N^{-}$, $f_{i}$ and $\lambda$. An upper bound for the exponent $r$ is the product $m\cdot e$. In general the bound $m\cdot e$ might be smaller than the upper bound computed in Theorem \ref{corollary:bound_for_a_0_nonzero} and Corollary \ref{corollary:bound_when_a_0_is_0}. We also use specific labels to indicate different prime ideals $\lambda$ that occur in the factorization of $\ell\mathcal{O}_{f_{i}}$. These labels are described in the MONTES package documentation\footnote{\url{http://www-ma4.upc.edu/~guardia/MontesAlgorithm.html}}. Hence, in the column labelled by $\lambda$ we use the notation $\lambda_{j}$ to denote a specific prime ideal with respect to the MONTES labelling. Similarly in the column ''form'' we let $f_{i}$ denote a specific newform that will appear.

\begin{example}
In Table \ref{table:data_example} we describe an example of a typical row of data in our congruence database. We read from it that a newform $f_{1}\in\mathcal{S}_{2}(2651)^{\mathrm{new}}$ is congruent to the Eisenstein series $E_{1,2651}$ modulo a power $\lambda_{1}^{2}$, where the ideal $\lambda_{1}$ is of residue characteristic $5$ and its ramification degree $e$ above $\ell=5$ equals $2$. Field degree $[K_{f_{1}}:\mathbb{Q}]$ is $35$ and $\mathcal{O}_{f_{1}}/\lambda_{1}=\mathbb{F}_{5}$. Theoretical upper bound for $r$ is $m\cdot e=4$ but our congruence appears only with the maximal exponent $r=2$.
{\def\arraystretch{1.1}
\begin{table}[ht]
\setlength{\tabcolsep}{7pt}
\begin{center}
\csvreader[tabular=c|c|c|c|c|c|c|c|c|c|c|c,
    table head= $N$ & $N^-$ & $N^+$ & $k$ & $\ell$ & $m$ &  form & $\lambda$ & $r$ & $e$ & $f$ & $d$ \\\hline,
    late after line=\\]%
{przyklad.csv}{n=\nnumber,nm=\nmnumber,np=\npnumber,fnumber=\fnumber, idnumber=\idnumber, r=\rnumber, k=\knumber, l=\lnumber, e=\enumber, f=\fnumber, d=\dnumber,m=\mnumber }%
{ \nnumber & \nmnumber & \npnumber & \knumber & \lnumber & \mnumber & $f_{\fnumber}$ & $\lambda_{\idnumber}$ & \rnumber & \enumber & \fnumber & \dnumber } %
\end{center}
\caption{Typical row of data}\label{table:data_example}
\end{table}}

\end{example}
 
\begin{example}
In Table \ref{table:high_exponents} we present for each pair $(r,\ell)$ one congruence for which $r$ is maximal in the whole range described in Table \ref{table:range}. In case there were more than one suitable pair $(r,\ell)$ we chose a specific pair and some $k$.
Moreover, we sort the data by descending value of $r$. 

{\def\arraystretch{1.1}
\begin{table}[htb]
\begin{center}
	\setlength{\tabcolsep}{7pt}
\csvreader[tabular=c|c|c|c|c|c|c|c|c|c|c|c|c,
    table head= & $N$ & $N^-$ & $N^+$ & $k$ & $\ell$ & $m$ &  form & $\lambda$ & $r$ & $e$ & $f$ & $d$ \\\hline,
    late after line=\\]%
{high_congruences.csv}{n=\nnumber,nm=\nmnumber,np=\npnumber,fnumber=\fnumber, idnumber=\idnumber, r=\rnumber, k=\knumber, l=\lnumber, e=\enumber, f=\fnumber, d=\dnumber,m=\mnumber }%
{\thecsvrow & \nnumber & \nmnumber & \npnumber & \knumber & \lnumber & \mnumber & $f_{\fnumber}$ & $\lambda_{\idnumber}$ & \rnumber & \enumber & \fnumber & \dnumber } %
\end{center}
\caption{Congruences that satisfy $r>2$ and $m>1$, one for each pair $(r,\ell)$}\label{table:high_exponents}
\end{table}}
\end{example}

\begin{example}
In Table \ref{table:ramified_congruences_1} we describe some examples of congruences that satisfy the non-trivial bound $r\leq e$ with $m>1$. We refer to Corollary \ref{corollary:ramification_bounds} for a precise statement of our observation.

{\def\arraystretch{1.1}
\begin{table}[htb]
\begin{center}
	\setlength{\tabcolsep}{7pt}
\csvreader[tabular=c|c|c|c|c|c|c|c|c|c|c|c|c,
    table head= & $N$ & $N^-$ & $N^+$ & $k$ & $\ell$ & $m$ &  form & $\lambda$ & $r$ & $e$ & $f$ & $d$ \\\hline,
    late after line=\\]%
{duze_e_duze_m_lower.csv}{n=\nnumber,nm=\nmnumber,np=\npnumber,fnumber=\fnumber, idnumber=\idnumber, r=\rnumber, k=\knumber, l=\lnumber, e=\enumber, f=\fnumber, d=\dnumber,m=\mnumber }%
{\thecsvrow & \nnumber & \nmnumber & \npnumber & \knumber & \lnumber & \mnumber & $f_{\fnumber}$ & $\lambda_{\idnumber}$ & \rnumber & \enumber & \fnumber & \dnumber } %
\end{center}
\caption{Exemplary congruences that satisfy conditions: $e>1$, $m>1$, $\ell>3$}\label{table:ramified_congruences_1}
\end{table}}
\end{example}

\section{Summary of computational results}\label{sec:Computational_results}
In this paragraph we summarize the large scale numerical computations that established the existence of congruences for square-free levels $N$ and weights $k$ in the range predicted by Table \ref{table:range}. We will say that \textit{there exists a congruence that satisfies} a condition $\mathcal{W}=\mathcal{W}(r,d,e,f,N^{-},N^{+},\ell,m)$ if we can find a weight $k$ and level $N$ such that there exists a newform $f\in\mathcal{S}_{k}(N)^{\mathrm{new}}$ and an Eisenstein eigenform $E\in\mathcal{E}_{k}(N)$ that satisfy (\ref{equation:explicit_congruence_general}) for an ideal $\lambda\in\mathcal{O}_{f}$ and a positive integer $r$ and such that the values of $r,d,e,f,N^{-},N^{+},\ell$ and $m$ associated with this congruence satisfy the condition $\mathcal{W}$.

\begin{corollary}
Let $N$ be a square-free number depending on the weight as described in Table \ref{table:range}. In Table \ref{table:summary_1} we present the number of different congruences of type (\ref{equation:explicit_congruence_general}) that can be found in the presented range.
In the column denoted by $r\geq 0$ we count the number of pairs $(f,\lambda)$ returned by Algorithm \ref{sec:algorithm}. In the column ''$r>0$'' we count the number of congruences, in the column ''$m\cdot e=r>0$'' we count the number of congruences with maximal exponent $r=m\cdot e$ and the last column has a similar meaning.
\end{corollary}
{\def\arraystretch{1.1}
\begin{table}[htb]
\begin{center}
		\setlength{\tabcolsep}{7pt}
\csvreader[tabular=c|c|c|c|c,
    table head= $k$ & $r\geq 0$  & $r>0$  & $\protect{m\cdot e =r > 0}$ & $\protect{m\cdot e >r >0}$ \\\hline,
    late after line=\\]%
{summary_cong1.csv}{k=\knumber,rgeq0=\rgeq,rgt0=\rgt,meeqr=\meeqr,megtr=\megtr}%
{ \knumber & \rgeq & \rgt & \meeqr & \megtr} %
\end{center}
\caption{Number of congruences of type (\ref{equation:explicit_congruence_general}) for fixed values of  $k$.}\label{table:summary_1}
\end{table}}

\begin{corollary}\label{corollary:ramification_bounds}
For $(N,k)$ from range in Table \ref{table:range} there exists 96 congruences that satisfy $e>1$, $m>1$ and $\ell>3$. Except for the cases described in Table \ref{table:exceptional_ramification} we have the bound $r\leq e$.
\end{corollary}

{\def\arraystretch{1.1}
\begin{table}[htb]
\begin{center}
\setlength{\tabcolsep}{7pt}
\csvreader[tabular=c|c|c|c|c|c|c|c|c|c|c|c|c,
    table head= & $N$ & $N^-$ & $N^+$ & $k$ & $\ell$ & $m$ &  form & $\lambda$ & $r$ & $e$ & $f$ & $d$ \\\hline,
    late after line=\\]%
{exceptional_r_gt_r_all_range.csv}{n=\nnumber,nm=\nmnumber,np=\npnumber,fnumber=\fnumber, idnumber=\idnumber, r=\rnumber, k=\knumber, l=\lnumber, e=\enumber, f=\fnumber, d=\dnumber,m=\mnumber }%
{\thecsvrow & \nnumber & \nmnumber & \npnumber & \knumber & \lnumber & \mnumber & $f_{\fnumber}$ & $\lambda_{\idnumber}$ & \rnumber & \enumber & \fnumber & \dnumber } %
\end{center}
\caption{Congruences that satisfy $e>1$, $m>1$, $\ell>3$ and $r>e$.}\label{table:exceptional_ramification}
\end{table}}

\begin{remark}
Corollary \ref{corollary:ramification_bounds} extends similar computations performed in \cite{Naskrecki_Springer} for prime levels $N$ and weight $k=2$. It was verified there that for primes $N\leq 13009$ the property $r\leq e$ holds for all $\ell>3$ and $e>1$. It is an open question if there are infinitely many such congruences for all possible ranges of $N$ and $k$.
\end{remark}

\begin{corollary}\label{corollary:large_degree_field}
Let $k=2$. For $N\leq 4559$ square-free and for any $d\leq 222$ we found congruences (\ref{equation:explicit_congruence_general}) if $d\notin D$ where
\begin{equation*}
\begin{split}
D=\{169,175,178,192,197,204,207,208,211,\\
214,215,216,217,218,219,220,221\}.
\end{split}
\end{equation*}
\end{corollary}
\begin{remark}
In \cite{Dieulefait_Jimenez_Ribet} the authors study the existence of newforms $f$ with large degree coefficient field $K_{f}$. The computations from Corollary \ref{corollary:large_degree_field} and Table \ref{table:large_degree} suggest that we can both find newforms that have large degree of $K_{f}$ and that are congruent to an Eisenstein eigenform. In Figure \ref{figure:dependens_on_degree_growth} we show that the growth of $d$ as a function of least $N$ is roughly a linear function. The way we present data in Table \ref{table:large_degree} is as follows: we assume $N^{-}=1$, in the $i$-th row we present a congruence such that $d\geq 10i$ for the least possible $N$. All values of $N$ that we found are prime numbers.
\end{remark}

{\def\arraystretch{1.1}
\begin{table}[htb]
\begin{center}
\setlength{\tabcolsep}{7pt}
\csvreader[tabular=c|c|c|c|c|c|c|c|c|c|c|c|c,
    table head= & $N$ & $N^-$ & $N^+$ & $k$ & $\ell$ & $m$ &  form & $\lambda$ & $r$ & $e$ & $f$ & $d$ \\\hline,
    late after line=\\]%
{high_degrees.csv}{n=\nnumber,nm=\nmnumber,np=\npnumber,fnumber=\fnumber, idnumber=\idnumber, r=\rnumber, k=\knumber, l=\lnumber, e=\enumber, f=\fnumber, d=\dnumber,m=\mnumber }%
{\thecsvrow & \nnumber & \nmnumber & \npnumber & \knumber & \lnumber & \mnumber & $f_{\fnumber}$ & $\lambda_{\idnumber}$ & \rnumber & \enumber & \fnumber & \dnumber } %
\end{center}
\caption{Selected congruences sorted by the degree $d$.}\label{table:large_degree}
\end{table}}

\begin{figure}
\begin{center}
\includegraphics{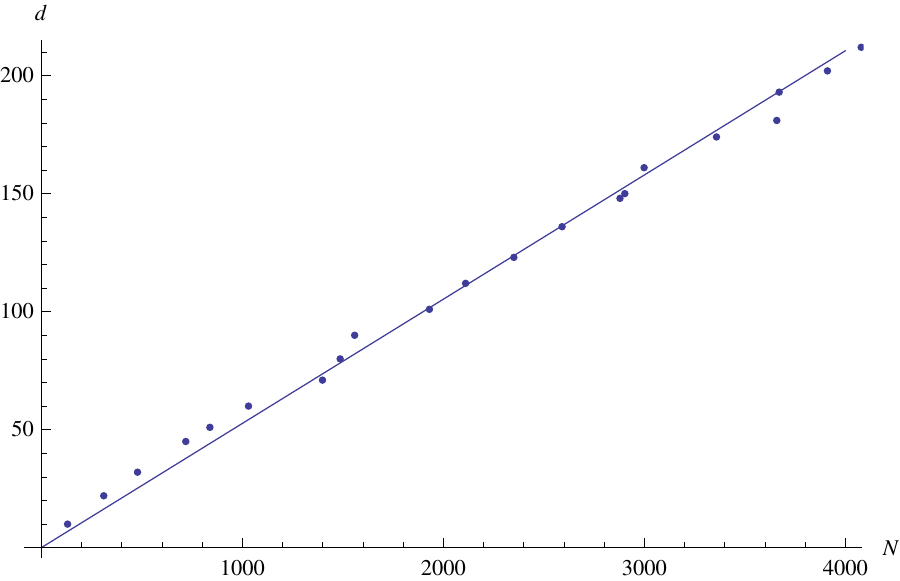}
\end{center}
\caption{Growth of degree $d$ as a function of least level $N$ for data from Table \ref{table:large_degree}.}\label{figure:dependens_on_degree_growth}
\end{figure}

\begin{corollary}
For $k=2$ and level $N$ less or equal to $4559$ there exist a congruence for any level except $N=13$ or $22$, for which the space $\mathcal{S}_{2}(N)^{\mathrm{new}}$ is zero.
\end{corollary}

\begin{corollary}
For $k=2$ and $N^{-}>1$ there exists $54077$ congruences for levels $N\leq 4559$ and $8860$ congruences for $N^{-}=1$ and levels $N\leq 4559$.
\end{corollary}
\begin{remark}
In \cite[Theorem 4.1.2]{Yoo_thesis} it is assumed that either $N_{-}=1$ and the number of prime divisors of $N$ is odd and $\ell \mid \phi(N)$ or the number of prime divisors of $N$ is even and $N^{-}$ is a prime number such that $\modulo{N^{-}}{-1}{\ell}$.
In several cases described in the above corollary the assumptions of \cite[Theorem 4.1.2]{Yoo_thesis} are satisfied. In those cases we obtain a congruence for the coefficients $a_{p}$ with $p\mid N$, which is not assumed in \cite[Theorem 4.1.2]{Yoo_thesis}. Moreover, some examples of the previous corollary suggest that the assumptions of \cite[Theorem 4.1.2]{Yoo_thesis} can be made weaker.
\end{remark}

\begin{corollary}
Let $(N,k)$ be a pair of integers that fit into the range of Table \ref{table:range}. In Table \ref{table:high_residue_degree} we present the number of corresponding congruences (abbreviated n.c. in the table) with $f>2$.
\end{corollary}

{\def\arraystretch{1.1}
\begin{table}[htb]
\begin{center}
\setlength{\tabcolsep}{7pt}
\begin{tabular}{c|c|c|c|c|c|c|c|c|c|c|c|c}
$k$     & 2 & 4 & 6 & 8 & 10 & 12 & 14 & 16 & 18 & 20 & 22 & 24 \\
\hline
n.c.   & 993 & 177 & 20 & 4 & 0 & 0 & 0 & 2 & 4 & 2 & 0 & 0
\end{tabular}
\end{center}
\caption{Weight $k$ and number of congruences that satisfy $f>2$.}\label{table:high_residue_degree}
\end{table}}

\begin{corollary}
For $N\leq 4559$ and $k=2$ there exist $30$ congruences that satisfy $e=17$ and $\ell=2$. In that range there is no congruence such that the ramification exponent $e$ is larger than $17$.
\end{corollary}

\begin{corollary}
Let $(N,k)$ be the numbers from the range in Table \ref{table:range}. Then in the described ranges there is an appropriate number of congruence (n.c.) that satisfy the condition $\ell\mid N$. Values are presented in Table \ref{table:ell_divides_N_congr}
\end{corollary}

{\renewcommand{\arraystretch}{1.1}
\begin{table}[htb]
\begin{center}
\setlength{\tabcolsep}{4pt}
\begin{tabular}{c|c|c|c|c|c|c|c|c|c|c|c|c}
$k$     & 2 & 4 & 6 & 8 & 10 & 12 & 14 & 16 & 18 & 20 & 22 & 24 \\
\hline
n.c. & 27771 & 4839 & 1366 & 1070 & 609 & 583 & 605 & 708 & 726 & 1323 & 990 & 1033
\end{tabular}
\end{center}
\caption{Weight $k$ and the number of congruences such that $\ell\mid N$.}\label{table:ell_divides_N_congr}
\end{table}}

\begin{corollary}
	Let $k=2$ and $N\leq 4559$. There are congruences for all prime characteristic $\ell\leq 2273$ except for the set
	\begin{equation*}
	\begin{split}
	\{353, 389, 457, 463, 523, 541, 569, 571, 587, 599, 613, 617, 631, 643,647, 677, 701,\\
	733, 757, 769, 773, 787, 797, 821, 823, 827, 839, 857, 859, 863, 881,  887, 907, 929,\\ 941, 947, 971, 977, 983, 991, 1021, 1051, 1061, 1091, 1097, 1109, 1117, 1151,\\
	1153, 1163, 1171, 1181, 1187, 1193, 1201, 1213, 1217, 1231, 1237, 1249, 1259,\\
	1277, 1279, 1283, 1291, 1297, 1301, 1303, 1307, 1319, 1321, 1327, 1361, 1367,\\
	1373, 1381, 1399, 1423, 1427, 1429, 1433, 1447, 1453, 1459, 1471, 1483, 1487,\\ 
	1489, 1493, 1523, 1531, 1543, 1549, 1553, 1567, 1571, 1579, 1597, 1607, 1609,\\
	1613, 1619, 1621, 1627, 1637, 1657, 1663, 1667, 1669, 1693, 1697, 1699, 1709,\\
	1721, 1723, 1741, 1747, 1753, 1759, 1777, 1783, 1787, 1789, 1801, 1823, 1831,\\
	1847, 1861, 1867, 1871, 1873, 1877, 1879, 1907, 1913, 1933, 1949, 1951, 1979,\\
	1987, 1993, 1997, 1999, 2011, 2017, 2027, 2029, 2053, 2081, 2083, 2087, 2089,\\
	2099, 2111, 2113, 2131, 2137, 2143, 2153, 2161, 2179, 2203, 2207, 2213, 2221,\\
	2237, 2239, 2243, 2251, 2267, 2269\}.
	\end{split}
	\end{equation*}
\end{corollary}

\section*{Acknowledgements}
The author would like to thank Wojciech Gajda for his excellent supervision of author's Ph.D. project. He also thanks Xavier Guitart, Kimball Martin, Hwajong Yoo and Gabor Wiese for reading the earlier version of the manuscript and their valuable comments. Finally he thanks both anonymous referees for their useful comments and corrections. The author was supported by the Polish National Science Centre research grant 2012/05/N/ST1/02871. This paper is partially based on the results obtained in the author's Ph.D. thesis \cite{Naskrecki_thesis}.


\begin{thebibliography}{10}
	
	\bibitem{Atkin_Lehner}
	A.~O.~L. Atkin and J.~Lehner.
	\newblock Hecke operators on {$\Gamma _{0}(m)$}.
	\newblock {\em Math. Ann.}, 185:134--160, 1970.
	
	\bibitem{BKK}
	T.~Berger, K.~Klosin, and K.~Kramer.
	\newblock On higher congruences between automorphic forms.
	\newblock {\em Math. Res. Lett.}, 21(1):71--82, 2014.
	
	\bibitem{Billerey_Menares_2}
	N.~Billerey and R.~Menares.
	\newblock On the modularity of reducible {${\rm mod}\, l$} {G}alois
	representations.
	\newblock {\em Math. Res. Lett.}, 23(1):15--41, 2016.
	
	\bibitem{Billerey_Menares_1}
	N.~Billerey and R.~Menares.
	\newblock Strong modularity of reducible {G}alois representations.
	\newblock {\em Trans. Amer. Math. Soc.}, 370(2):967--986, 2018.
	
	\bibitem{MAGMA}
	W.~Bosma, J.~Cannon, and C.~Playoust.
	\newblock The {M}agma algebra system. {I}. {T}he user language.
	\newblock {\em J. Symbolic Comput.}, 24(3-4):235--265, 1997.
	\newblock Computational algebra and number theory ({L}ondon, 1993).
	
	\bibitem{Kiming_mod_powers}
	I.~Chen, I.~Kiming, and J.~B. Rasmussen.
	\newblock On congruences mod {$p^m$} between eigenforms and their attached
	{G}alois representations.
	\newblock {\em J. Number Theory}, 130(3):608--619, 2010.
	
	\bibitem{Cremona_Modular}
	J.~Cremona.
	\newblock {\em Algorithms for modular elliptic curves}.
	\newblock Cambridge University Press, Cambridge, second edition, 1997.
	
	\bibitem{Diamond_Shurman}
	F.~Diamond and J.~Shurman.
	\newblock {\em A first course in modular forms}, volume 228 of {\em Graduate
		Texts in Mathematics}.
	\newblock Springer-Verlag, New York, 2005.
	
	\bibitem{Dieulefait_Jimenez_Ribet}
	L.~V. Dieulefait, J.~Jim\'enez~Urroz, and K.~A. Ribet.
	\newblock Modular forms with large coefficient fields via congruences.
	\newblock {\em Res. Number Theory}, 1:Art. 2, 14, 2015.
	
	\bibitem{Dummigan_Fretwell}
	N.~Dummigan and D.~Fretwell.
	\newblock {R}amanujan-style congruences of local origin.
	\newblock {\em J. Number Theory}, 143:248--261, 2014.
	
	\bibitem{Montes}
	J.~Gu\`ardia, J.~Montes, and E.~Nart.
	\newblock Higher {N}ewton polygons and integral bases.
	\newblock {\em J. Number Theory}, 147:549--589, 2015.
	
	\bibitem{Hsu}
	C.~{Hsu}.
	\newblock Higher congruences between newforms and {E}isenstein series of
	squarefree level.
	\newblock {\em ArXiv e-prints}, 2017.
	\newblock arXiv:1706.05589.
	
	\bibitem{Katz_Torsion}
	N.~Katz.
	\newblock Galois properties of torsion points on abelian varieties.
	\newblock {\em Invent. Math.}, 62(3):481--502, 1981.
	
	\bibitem{Lecouturier}
	E.~{Lecouturier}.
	\newblock Higher {E}isenstein elements, higher {E}ichler formulas and rank of
	{H}ecke algebras.
	\newblock {\em ArXiv e-prints}, 2017.
	\newblock arXiv: 1709.09114.
	
	\bibitem{Martin}
	K.~{Martin}.
	\newblock The basis problem revisited.
	\newblock {\em ArXiv e-prints}, 2018.
	\newblock arXiv:1804.04234.
	
	\bibitem{Mazur_modular}
	B.~Mazur.
	\newblock Modular curves and the {E}isenstein ideal.
	\newblock {\em Inst. Hautes \'Etudes Sci. Publ. Math.}, (47):33--186 (1978),
	1977.
	
	\bibitem{Naskrecki_Springer}
	B.~Naskręcki.
	\newblock On higher congruences between cusp forms and {E}isenstein series.
	\newblock In {\em Computations with modular forms}, volume~6 of {\em Contrib.
		Math. Comput. Sci.}, pages 257--277. Springer, Cham, 2014.
	
	\bibitem{Naskrecki_thesis}
	B.~Naskręcki.
	\newblock Ranks in families of elliptic curves and modular forms.
	\newblock {\em Adam Mickiewicz University}, 2014.
	\newblock Ph.D. thesis.
	
	\bibitem{Sadek_torsion}
	M.~Sadek.
	\newblock On elliptic curves whose conductor is a product of two prime powers.
	\newblock {\em Math. Comp.}, 83(285):447--460, 2014.
	
	\bibitem{Sturm}
	J.~Sturm.
	\newblock On the congruence of modular forms.
	\newblock In {\em Number theory ({N}ew {Y}ork, 1984--1985)}, volume 1240 of
	{\em Lecture Notes in Math.}, pages 275--280. Springer, Berlin, 1987.
	
	\bibitem{Yoo_thesis}
	H.~Yoo.
	\newblock {\em Modularity of residually reducible {G}alois representations and
		{E}isenstein ideals}.
	\newblock ProQuest LLC, Ann Arbor, MI, 2013.
	\newblock Thesis (Ph.D.)--University of California, Berkeley.
	
	\bibitem{Yoo_multiplicity}
	H.~Yoo.
	\newblock The index of an {E}isenstein ideal and multiplicity one.
	\newblock {\em Math. Z.}, 282(3-4):1097--1116, 2016.
	
\end{thebibliography}
\def\polhk#1{\setbox0=\hbox{#1}{\ooalign{\hidewidth
			\lower1.5ex\hbox{`}\hidewidth\crcr\unhbox0}}}

\end{document}